\documentclass[a4paper,12pt]{amsart}

\usepackage[T1]{fontenc}
\usepackage[utf8]{inputenc}
\usepackage{lmodern}
\usepackage[english]{babel}

\usepackage[%
	left=2.5cm,       
	right=2.5cm,      
	top=3.5cm,        
	bottom=3.5cm,     
	heightrounded,    
	bindingoffset=0mm 
]{geometry}

\usepackage{amsmath}
\usepackage{amssymb}
\usepackage{mathtools}
\usepackage{mathrsfs}

\numberwithin{equation}{section}

\theoremstyle{plain}
	\newtheorem{theorem}{Theorem}[section]
	\newtheorem{lemma}[theorem]{Lemma}
	\newtheorem{proposition}[theorem]{Proposition}
	\newtheorem{corollary}[theorem]{Corollary}

\theoremstyle{definition}
	\newtheorem{definition}[theorem]{Definition}

\newcommand{\N}{\mathbb{N}}
\newcommand{\Z}{\mathbb{Z}}
\newcommand{\R}{\mathbb{R}}
\newcommand{\haus}{\mathcal{H}}
\newcommand{\de}{\partial}

\DeclareMathOperator{\dist}{dist}

\DeclareMathOperator{\conv}{conv}
					    
\DeclarePairedDelimiter{\scalar}{<}{>}                                     
\DeclarePairedDelimiter{\set}{\{}{\}}

\mathchardef\ordinarycolon\mathcode`\:
\mathcode`\:=\string"8000
\begingroup \catcode`\:=\active
  \gdef:{\mathrel{\mathop\ordinarycolon}}
\endgroup

\newcommand{\mres}{\mathbin{\vrule height 1.6ex depth 0pt width
0.13ex\vrule height 0.13ex depth 0pt width 1.3ex}}

\usepackage{bookmark}
\usepackage{hyperref} 

\usepackage[initials]{amsrefs}

\usepackage{tikz}
\usepackage{graphicx}
\usepackage[font=small]{caption}

\newcommand{\closure}[2][3]{%
  {}\mkern#1mu\overline{\mkern-#1mu#2}}
  
\usepackage{enumerate}

\allowdisplaybreaks

\begin{document}

\title[Monotonicity of perimeter of convex bodies]{On the monotonicity of perimeter of convex bodies}

\author[G. Stefani]{Giorgio Stefani}

\address{Scuola Normale Superiore, Piazza Cavalieri 7, 56126 Pisa, Italy}

\email{giorgio.stefani@sns.it}

\dedicatory{Dedicated to Francesco Leonetti on the occasion of his 60th birthday}

\date{November 6, 2016}



\begin{abstract}
Let $n\ge2$ and let $\Phi\colon\R^n\to[0,\infty)$ be a positively $1$-homogeneous and convex function. Given two convex bodies $A\subset B$ in $\R^n$, the monotonicity of anisotropic $\Phi$-perimeters holds, i.e.\ $P_\Phi(A)\le P_\Phi(B)$. In this note, we prove a quantitative lower bound on the difference of the $\Phi$-perimeters of $A$ and $B$ in terms of their Hausdorff distance.
\end{abstract}

\maketitle

\section{Introduction}

Let $n\ge2$ and let $A,B\subset\R^n$ be two convex bodies (i.e., compact convex sets with non-empty interior). If $A\subset B$, then the monotonicity of perimeters holds, i.e.\
\begin{equation}\label{eq:perim_ineq}
\haus^{n-1}(\de A)\le\haus^{n-1}(\de B).
\end{equation}
Here and in the following, for all $s\ge0$ we let $\haus^s$ be the $s$-dimensional Hausdorff measure (in particular, $\haus^0$ is the counting measure). Moreover, if $E\subset\R^n$ is a $k$-dimensional convex body, with $1\le k\le n$, we let $\de E$ be its boundary, which is a set of Hausdorff dimension $k-1$.

Inequality~\eqref{eq:perim_ineq} is well-known and dates back to the ancient Greek (Archimedes himself took it as a postulate in his work on the sphere and the cylinder,~\cite{archimedes04}*{p.~36}). Various proofs of~\eqref{eq:perim_ineq} are possible: via the Cauchy formula for the area surface of convex bodies or by the monotonicity property of mixed volumes,~\cite{bonnesen&fenchel87}*{\S7}, by the Lipschitz property of the projection on a convex closed set,~\cite{buttazzoetal95}*{Lemma~2.4}, or by the fact that the perimeter is decreased under intersection with half-spaces,~\cite{maggi12}*{Exercise~15.13}. 

Lower bounds for the deficit $\delta(B,A)=\haus^{n-1}(\de B)-\haus^{n-1}(\de A)$ with respect to the Hausdorff distance $h(A,B)$ of $A$ and $B$ have been recently established for $n=2,3$ in~\cites{lacivita&leonetti08,carozzaetal15,carozzaetal16}. The case $n=2$ was treated for the first time in~\cite{lacivita&leonetti08}, and was subsequently improved in~\cite{carozzaetal15} to the following inequality
\begin{equation}\label{eq:planar_estimate}
\haus^1(\de A)+\dfrac{2h(A,B)^2}{\sqrt{\left(\frac{\haus^1(B\cap L)}{2}\right)^2+h(A,B)^2}+\frac{\haus^1(B\cap L)}{2}}\le\haus^1(\de B),
\end{equation}
where $L=\set{x\in\R^2 : \scalar*{b-a,x-a}=0}$, with $a\in A$ and $b\in B$ such that $|a-b|=h(A,B)$. The case $n=3$ was studied in~\cite{carozzaetal16}, where the authors proved the following inequality
\begin{equation}\label{eq:R3_estimate}
\haus^2(\de A)+\dfrac{\pi d h(A,B)^2}{\sqrt{d^2+h(A,B)^2}+d}\le\haus^2(\de B),
\end{equation}
with $h(A,B)$, $a\in A$ and $b\in B$ as above and $d=\dist(a,\de B\cap\de H)$, where $H=\set{x\in\R^3 : \scalar*{b-a,x-a}\le0}$. Inequalities~\eqref{eq:planar_estimate} and~\eqref{eq:R3_estimate} are sharp, in the sense that they are equalities at least in one case, see~\cites{carozzaetal15,carozzaetal16}. Inequality~\eqref{eq:R3_estimate}, however, does not seem to be the correct generalization of inequality~\eqref{eq:planar_estimate} to the case $n=3$, because of the distance $d=\dist(a,\de B\cap\de H)$ replacing the bigger radius $r=\sqrt{\haus^2(B\cap\de H)/\pi}$.   

Inequality~\eqref{eq:perim_ineq} naturally generalizes to the anisotropic (Wulff) perimeter. Precisely, given a positively $1$-homogeneous convex function $\Phi\colon\R^n\to[0,\infty)$, if $A\subset B$ are two convex bodies in $\R^n$, then
\begin{equation}\label{eq:phi-perim_ineq}
P_\Phi(A)\le P_\Phi(B).
\end{equation} 
Here $P_\Phi(E)$ denotes the \emph{anisotropic $\Phi$-perimeter} of a convex body $E\subset\R^n$ and is defined as
\begin{equation*}
P_\Phi(E)=\int_{\de E} \Phi(\nu_E)\ d\haus^{n-1},
\end{equation*}   
where $\nu_E\colon\de E\to\R^n$ is the inner unit normal of $E$ (defined $\haus^{n-1}$-a.e. on $\de E$). Clearly, when $\Phi(x)=|x|$ for all $x\in\R^n$, then $P_\Phi(E)=\haus^{n-1}(\de E)$, the Euclidean perimeter of $E$. The $\Phi$-perimeter obeys the scaling law $P_\Phi(\lambda E)=\lambda^{n-1}P_\Phi(E)$, $\lambda>0$, and it is invariant under translations. However, at variance with the Euclidean perimeter, $P_\Phi$ is not invariant by the action of $O(n)$, or even of $SO(n)$, and in fact it may even happen that $P_\Phi(E)\ne P_\Phi(\R^n\setminus E)$, provided that $\Phi$ is not symmetric with respect to the origin. 

Similarly to inequality~\eqref{eq:perim_ineq}, inequality~\eqref{eq:phi-perim_ineq} is a consequence of the Cauchy formula for the anisotropic perimeter or of the monotonicity property of mixed volumes,~\cite{bonnesen&fenchel87}*{\S7,~\S8}, or of the fact that the anisotropic perimeter is decreased under intersection with half-spaces,~\cite{maggi12}*{Remark~20.3}.

The aim of this note is to establish a lower bound for the anisotropic deficit $\delta_\Phi(B,A)=P_\Phi(B)-P_\Phi(A)$ with respect to the Hausdorff distance $h(A,B)$ of $A$ and $B$. Before stating our main result, we need some preliminaries. Here and in the rest of the paper, we let
\begin{equation*}
\mathbb{S}^{n-1}=\set*{x\in\R^n : |x|=1}, \qquad \nu^\perp=\set*{x\in\R^n: x\cdot\nu=0} \quad \forall\nu\in\mathbb{S}^{n-1}.
\end{equation*} 

\begin{definition}[Admissible $\Phi$]\label{def:admissible}
Let $n\ge2$ and let $\Phi\colon\R^n\to[0,\infty)$ be a positively $1$-homogeneous convex function. We say that $\Phi$ is \emph{admissible} if, for each $\nu\in\mathbb{S}^{n-1}$, there exist two functions $g_\nu\colon[0,\infty)^2\to[0,\infty)$ and $\phi_\nu\colon\nu^\perp\to[0,\infty)$ such that
\begin{enumerate}[(i)]
\item\label{item:1} $g_\nu$ is non-identically zero, positively $1$-homogeneous, convex and $s\mapsto g_\nu(s,t)$ is non-decreasing for each fixed $t\in[0,\infty)$;  
\item\label{item:2} $\phi_\nu$ is positively $1$-homogeneous, convex and coercive on $\nu^\perp$, i.e.\ $\phi_\nu(z)>0$ for all $z\in\nu^\perp$, $z\ne0$;
\item\label{item:3} for all $x\in\R^n$ with $x\cdot\nu\ge0$, it holds
\begin{equation}\label{eq:def_g_phi_nu}
\Phi(x)\ge g_\nu(\phi_\nu(x-(x\cdot\nu)\nu),x\cdot\nu). 
\end{equation} 
\end{enumerate}
\end{definition}

If $\Phi$ is positively $1$-homogeneous, convex and coercive on $\R^n$, i.e.\ $\Phi(x)>0$ for all $x\in\R^n$, $x\ne0$, then $\Phi$ is admissible, since the choice $\phi_\nu(z)=|z|$, $z\in\nu^\perp$, and $g_\nu(s,t)=c\sqrt{s^2+t^2}$, $s,t\ge0$, with $c=\min\set*{\Phi(x):|x|=1}$, is possible for all $\nu\in\mathbb{S}^{n-1}$ (although not the best one for special directions in general). 

We can now state our main result, which is contained in the following theorem. Here and in the rest of the paper, for each $\nu\in\mathbb{S}^{n-1}$, we let $W_\nu\subset\nu^\perp$ be the Wulff shape associated with $\phi_\nu$ in $\nu^\perp$, i.e.\ 
\begin{equation}\label{eq:def_W_nu}
W_\nu=\set*{z\in\nu^\perp : \phi_\nu^*(z)\le1},
\end{equation}
where $\phi_\nu^*\colon\nu^\perp\to[0,\infty)$ is given by $\phi_\nu^*(z)=\sup\set*{z\cdot w : \phi_\nu(w)<1}$  for all $z\in\nu^\perp$. Moreover, for any $a\in\R$ we let $a^+=\max\set*{a,0}$. 

\begin{theorem}\label{th:main}
Let $n\ge2$ and let $\Phi\colon\R^n\to[0,\infty)$ be a positively $1$-homogeneous convex function which is admissible in the sense of Definition~\ref{def:admissible}. If $A\subset B$ are two convex bodies in $\R^n$, then
\begin{equation}\label{eq:main_estimate}
P_\Phi(A)+\haus^{n-1}(W_{\nu_H})r^{n-2}\big(g_{\nu_H}(h,r)-\Phi(\nu_H)r\big)^+\le P_\Phi(B),
\end{equation}
where $h=h(A,B)$ is the Hausdorff distance of $A$ and $B$ and
\begin{equation}\label{eq:main_defs_r_H_nu_H}
r=\sqrt[n-1]{\frac{\haus^{n-1}(B\cap\de H)}{\haus^{n-1}(W_{\nu_H})}}, \qquad H=\set{x\in\R^n : \scalar*{b-a,x-a}\le0}, \qquad \nu_H=\frac{a-b}{|a-b|},
\end{equation}
with $a\in A$ and $b\in B$ such that $|a-b|=h(A,B)$.
\end{theorem}

When $P_\Phi$ reduces to the Euclidean perimeter, Theorem~\ref{th:main} provides the correct generalization of inequality~\eqref{eq:planar_estimate} to higher dimensions $n\ge3$. Indeed, if $\Phi(x)=|x|$, $x\in\R^n,$ then, for all $\nu\in\mathbb{S}^{n-1}$, we have  
\begin{equation*}
g_\nu(s,t)=\sqrt{s^2+t^2},\ s,t\ge0, \qquad \phi_\nu(z)=|z|,\ z\in\nu^\perp,
\end{equation*} 
so $W_\nu=\mathbb{B}^n\cap\nu^\perp\equiv\mathbb{B}^{n-1}$ and $\haus^{n-1}(W_\nu)=\omega_{n-1}$, where $\mathbb{B}^k$ is the $k$-dimensional closed unit ball, $1\le k\le n$. We thus have the following result.

\begin{corollary}\label{corol:euclid_perim}
Let $n\ge2$. If $A\subset B$ are two convex bodies in $\R^n$, then
\begin{equation}\label{eq:coroll_euclid_perim}
\haus^{n-1}(\de A)+\frac{\omega_{n-1}r^{n-2}h^2}{\sqrt{h^2+r^2}+r}\le\haus^{n-1}(\de B),
\end{equation}
where $h=h(A,B)$ is the Hausdorff distance of $A$ and $B$ and
\begin{equation*}
r=\sqrt[n-1]{\frac{\haus^{n-1}(B\cap\de H)}{\omega_{n-1}}}, \qquad H=\set{x\in\R^n : \scalar*{b-a,x-a}\le0},
\end{equation*}
with $a\in A$ and $b\in B$ such that $|a-b|=h(A,B)$.
\end{corollary}

Inequality~\eqref{eq:main_estimate} is not sharp in general. On the other hand, if we assume that~\eqref{eq:def_g_phi_nu} holds as an equality for some $\nu\in\mathbb{S}^{n-1}$ and if we impose strict convexity and strict monotonicity to the corresponding~$g_\nu$, then inequality~\eqref{eq:main_estimate} becomes sharp. In fact, this case corresponds to the setting studied in~\cite{baer15} and it is not difficult to see that the convex bodies
\begin{equation*}
A=\set*{x\in\R^n : 0\le x\cdot\nu\le 1,\ x-(x\cdot\nu)\nu\in W_\nu}, \qquad B=A\cup\mathcal{C}(-\nu,W_\nu),
\end{equation*}
provide the desired configuration. Here and in the rest of the paper, $\mathcal{C}(p,S)$ denotes the cone with vertex the point $p\in\R^n$ and base the nonempty set $S\subset\R^n$, i.e.\ the union of all straight line segments joining $p$ with a point in $S$. 

As a consequence, inequality~\eqref{eq:coroll_euclid_perim} is sharp, but the reader can easily check this fact generalizing the examples given in~\cites{carozzaetal15,carozzaetal16} to higher dimensions.

\section{Proof of Theorem~\ref{th:main}}\label{sec:proofs}

In this section, we prove Theorem~\ref{th:main}. The main ingredient of the argument is given by Lemma~\ref{lemma:cones} below, which can be seen as a consequence of the anisotropic symmetrization techniques developed in~\cite{baer15}. Here we follow a more elementary approach modeled on the special geometry of cones. We will make use of the $(n-1)$-dimensional Wulff inequality and of the following form of Jensen’s inequality.

\begin{proposition}[Jensen inequality]\label{prop:jensen}
Let $(X,\mu,\mathcal{M})$ be a measure space with $\mu(X)<\infty$ and let $g\colon[0,\infty)^2\to[0,\infty)$ be a positively $1$-homogeneous convex function. Then, for all $\mu$-measurable functions $f_1,f_2\colon X\to[0,\infty)$, we have
\begin{equation*}
g\left(\int_X f_1\ d\mu,\int_X f_2\ d\mu\right)\le\int_X g(f_1,f_2)\ d\mu.
\end{equation*}
Moreover, if $g$ is strictly convex in either argument, then equality holds if and only if $f_1/f_2$ is constant $\mu$-a.e. on $X$.  
\end{proposition}

In the proof of Lemma~\ref{lemma:cones}, we will also need to conveniently approximate convex bodies by means of convex polytopes, i.e.\ convex bodies with polyhedral boundary. This approximation is contained in Lemma~\ref{lemma:approx} below, which we state and prove here for the reader's convenience. In the following, for any convex body $K\subset\R^n$, we set
\begin{equation*}
\mu_K=\nu_K\,\haus^{n-1}\,\mres\,\de K,
\end{equation*}
where $\nu_K$ is the inner unit normal of $K$.

\begin{lemma}[Approximation by convex polytopes]\label{lemma:approx}
Let $n\ge2$ and let $E$ be a convex body in $\R^n$. There exists a sequence $(C_k)_{k\in\N}$ of convex polytopes in $\R^n$ with the following properties:
\begin{equation}\label{eq:approx_haus_dist}
E\subset C_k\subset E+\lambda_k\mathbb{B}^n, \quad \lambda_k=h(C_k,E)\le\frac{1}{k},
\end{equation}
\begin{equation}\label{eq:approx_vol_perim}
\haus^{n}(C_k\setminus E)\to0, \quad \haus^{n-1}(\de C_k)\to\haus^{n-1}(\de E) \quad\text{as } k\to\infty,
\end{equation}
\begin{equation}\label{eq:approx_meas}
\mu_{C_k}\overset{*}{\rightharpoonup}\mu_E, \quad |\mu_{C_k}|\overset{*}{\rightharpoonup}|\mu_E| \quad\text{as } k\to\infty.
\end{equation}
\end{lemma}

\begin{proof}
For each $k\in\N$, let $Q_k=[0,\tfrac{1}{{k}\sqrt{n}}]^n$ and consider the family of cubes
\begin{equation*}
\mathcal{F}_k=\set*{z+Q_k : z\in\Z^n}.
\end{equation*}
Then define
\begin{equation*}
C_k=\conv\set*{Q\in\mathcal{F}_k : E\cap Q\ne\varnothing},
\end{equation*}
where $\conv S$ denotes the convex envelope of the set $S\subset\R^n$. By construction, $C_k$ is a convex polytope that satisfies~\eqref{eq:approx_haus_dist}. As a consequence, we have
\begin{equation*}
\haus^{n}(C_k\setminus E)\le\haus^{n}((E+\lambda_k\mathbb{B}^n)\setminus E)
\end{equation*}
and, by~\eqref{eq:perim_ineq},
\begin{equation*}
\haus^{n-1}(\de E)\le\haus^{n-1}(\de C_k)\le\haus^{n-1}(\de(E+\lambda_k\mathbb{B}^n)).
\end{equation*}
Thus~\eqref{eq:approx_vol_perim} follows from the Steiner formulas for outer parallel bodies, see~\cite{gruber07}*{Theorem~6.14}. Moreover, since $\chi_{C_k}\to\chi_E$ in $L^1$ by~\eqref{eq:approx_vol_perim}, by the divergence theorem we have
\begin{equation*}
\int_{\R^n}\phi\ d\mu_{C_k}=\int_{C_k}\nabla\phi\ dx\to\int_{E}\nabla\phi\ dx=\int_{\R^n}\phi\ d\mu_E
\end{equation*}
for all $\phi\in\mathscr{C}^1_c(\R^n)$. By the density of $\mathscr{C}^1_c(\R^n)$ into $\mathscr{C}^0_c(\R^n)$, we easily get $\mu_{C_k}\overset{*}{\rightharpoonup}\mu_E$. Finally, since $C_k\subset E+\mathbb{B}^n$ for all $k\in\N$, by~\cite{maggi12}*{Exercise~4.31} we also have $|\mu_{C_k}|\overset{*}{\rightharpoonup}|\mu_E|$ and~\eqref{eq:approx_meas} follows. 
\end{proof}

\begin{lemma}\label{lemma:cones}
Let $n\ge2$ and let $\Phi\colon\R^n\to[0,\infty)$ be a positively $1$-homogeneous convex function which is admissible in the sense of Definition~\ref{def:admissible}. Fix $\nu\in\mathbb{S}^{n-1}$ and let $E\subset\nu^\perp$ be a $(n-1)$-dimensional convex body with $0\in E$. Let $b\in\R^n$ such that $b=-h\nu$ for some $h>0$. We set $C=\mathcal{C}(b,E)$, $C_{lat}=\mathcal{C}(b,\de E)$ and we let $\nu_C$ be the inner unit normal of the cone $C$. Then
\begin{equation}\label{eq:lemma_cones}
\int_{C_{lat}}\Phi(\nu_C)\ d\haus^{n-1}\ge\haus^{n-1}(W_\nu)r^{n-2}g_\nu(h,r),
\end{equation}
where $r^{n-1}\haus^{n-1}(W_\nu)=\haus^{n-1}(E)$ and $W_\nu\subset\nu^\perp$ was defined in~\eqref{eq:def_W_nu}.
\end{lemma}

\begin{proof}
The case $n=2$ is easy and we leave it to the reader. Thus, in the following, we directly assume that $n\ge3$. 

The inner unit normal $\nu_E$ of $E$ is defined $\haus^{n-2}$-a.e. on $\de E$ and belongs to the tangent bundle of the hyperplane $\nu^\perp$. Therefore, when $\nu_E\in\mathbb{S}^{n-2}$ is defined, we can naturally identify it with a unit vector in $\R^n$ that we still denote by $\nu_E$.

\textsc{Step one}. Let us assume that $E$ is convex polytope with faces $F_1,\dots, F_m$ for some $m\ge1$. Then, for each $k=1,\dots,m$, $\nu_E$ is constant on $\mathring{F_k}$ (the interior of $F_k$ in the relative topology) and we set $\nu_E=\nu_k$ on $\mathring{F_k}$, for some $\nu_k\in\mathbb{S}^{n-1}$.

By definition of cone, $C_{lat}$ is the union of $m$ $(n-1)$-dimensional cones $\Delta_k=\mathcal{C}(b,F_k)$. Note that $\Delta_k$ is contained in a hyperplane $L_k$ and that $\mathring{\Delta}_k=\mathcal{C}(b,\mathring{F_k})$. Therefore, for each $k=1,\dots,m$, $\nu_C$ is constant on $\mathring{\Delta}_k$ and equals the unit normal to $L_k$ with sign chosen so that $\nu_C\cdot\nu>0$.

Note that, given any $x\in F_k$, $F_k$ is contained in the intersection $I_k$ of the hyperplanes $\nu^\perp$ and $x+\nu_k^\perp$ ($I_k$ is independent of the choice of $x\in F_k$). Thus, for each $k=1,\dots,m$, the height $t_k>0$ of the cone $\Delta_k$ is given by $t_k=d(b,I_k)$. Letting $d_k=d(0,I_k)$, we then have $t_k^2=h^2+d_k^2$ and $t_k\nu_C=h\nu_k+d_k\nu$ on $\mathring{\Delta}_k$. Moreover, given any $x\in\mathring{F_k}$, we have $d_k=|x\cdot\nu_k|=|x\cdot\nu_E(x)|$. 

In conclusion, we have
\begin{align*}
\int_{C_{lat}}\Phi(\nu_C)\ d\haus^{n-1}&=\sum_{k=1}^m\Phi(\nu_C|_{\mathring{\Delta}_k})\cdot\haus^{n-1}(\Delta
_k)\\
&=\sum_{k=1}^m\Phi\left(\frac{h}{t_k}\nu_k+\frac{d_k}{t_k}\nu\right)\cdot\frac{1}{n-1}\haus^{n-2}(F_k)\,t_k\\
&=\frac{1}{n-1}\sum_{k=1}^m\Phi(h\nu_k+d_k\nu)\cdot\haus^{n-2}(F_k)\\
&=\frac{1}{n-1}\int_{\de E}\Phi(h\nu_E(x)+|x\cdot\nu_E(x)|\nu)\ d\haus^{n-2}(x).
\end{align*}

\textsc{Step two}. By~\eqref{eq:def_g_phi_nu}, we have
\begin{equation*}
\int_{\de E}\Phi(h\nu_E(x)+|x\cdot\nu_E(x)|\nu)\ d\haus^{n-2}(x)\ge\int_{\de E}g_\nu(h\phi_\nu(\nu_E(x)),|x\cdot\nu_E(x)|)\ d\haus^{n-2}(x).
\end{equation*}
Now let $r>0$ be such that $r^{n-1}\haus^{n-1}(W_\nu)=\haus^{n-1}(E)$ as in the statement of the Lemma. By the Wulff inequality in $\nu^\perp$, we have $p_{\phi_\nu}(rW_\nu)\le p_{\phi_\nu}(E)$, where for all convex body $K\subset\nu^\perp$ we define
\begin{equation*}
p_{\phi_\nu}(K)=\int_{\de K}\phi_\nu(\nu_K)\ d\haus^{n-2}.
\end{equation*} 
Moreover, note that
\begin{equation*}
\frac{1}{n-1}\int_{\de E}|\nu_E(x)\cdot x|\ d\haus^{n-2}=\frac{1}{n-1}\sum_{k=1}^m d_k\haus^{n-2}(F_k)=\haus^{n-1}(E),
\end{equation*}
because, by the definitions in \textsc{Step one}, for each $k=1,\dots,m$, $d_k$ is exactly the height of the $(n-1)$-dimensional cone $\mathcal{C}(0,F_k)$ and $E=\cup_{k=1}^m\mathcal{C}(0,F_k)$ since $E$ is convex and $0\in E$.

Recalling that $g_\nu$ is increasing in the first argument and applying the Jensen inequality given in Proposition~\ref{prop:jensen}, we find
\begin{align*}
g_\nu\big(hp_{\phi_\nu}(rW_\nu),&(n-1)\haus^{n-1}(rW_\nu)\big)\le g_\nu\big(hp_{\phi_\nu}(E),(n-1)\haus^{n-1}(E)\big)\\
&=g_\nu\left(h\int_{\de E}\phi_\nu(\nu_E(x))\ d\haus^{n-2}(x),\int_{\de E}|\nu_E(x)\cdot x|\ d\haus^{n-2}(x)\right)\\
&\le\int_{\de E}g_\nu(\phi_\nu(h\nu_E(x)),|\nu_E(x)\cdot x|)\ d\haus^{n-2}(x).
\end{align*}
Therefore, since $p_{\phi_\nu}(W_\nu)=(n-1)\haus^{n-1}(W_\nu)$ and $g_\nu$ is positively $1$-homogeneous,
\begin{align*}
\int_{\de E}\Phi(h\nu_E(x)&+|x\cdot\nu_E(x)|\nu)\ d\haus^{n-2}(x)\ge g_\nu\big(hp_{\phi_\nu}(rW_\nu),(n-1)\haus^{n-1}(rW_\nu)\big)\\
&=g_\nu\big(hr^{n-2}(n-1)\haus^{n-1}(W_\nu),r^{n-1}(n-1)\haus^{n-1}(W_\nu)\big)\\
&=(n-1)\haus^{n-1}(W_\nu)r^{n-2}g_\nu(h,r).
\end{align*}
In conclusion, we get
\begin{align*}
\int_{C_{lat}}\Phi(\nu_C)\ d\haus^{n-1}&=\frac{1}{n-1}\int_{\de E}\Phi(h\nu_E(x)+|x\cdot\nu_E(x)|\nu)\ d\haus^{n-2}(x)\\
&\ge\haus^{n-1}(W_\nu)r^{n-2}g_\nu(h,r).
\end{align*}
This proves~\eqref{eq:lemma_cones} when $E$ is a convex polytope.

\textsc{Step Three}. Now let $E\subset\nu^\perp$ be any convex body and let $(E_k)_{k\in\N}$ be the sequence of convex polytopes approximating $E$ in $\nu^\perp$ given by Lemma~\ref{lemma:approx}. Letting $r_k,r>0$ be such that $r_k^{n-1}\haus^{n-1}(W_\nu)=\haus^{n-1}(E_k)$, we clearly have $r_k\to r$ as $k\to\infty$.

Let $C_k=\mathcal{C}(b,E_k)$ and $C_{k,lat}=\mathcal{C}(b,\de E_k)$ for each $k\in\N$. By~\eqref{eq:lemma_cones}, we have
\begin{equation*}
\int_{C_{k,lat}}\Phi(\nu_{C,k})\ d\haus^{n-1}\ge\haus^{n-1}(W_\nu)r_k^{n-2}g_\nu(h,r_k)
\end{equation*}
and thus, adding $\Phi(-\nu)\,\haus^{n-1}(E_k)$ to both sides, we find
\begin{equation}\label{eq:lemma_cones-k}
P_\Phi(C_k)\ge\haus^{n-1}(W_\nu)r_k^{n-2}g_\nu(h,r_k)+\Phi(-\nu)\,\haus^{n-1}(E_k).
\end{equation}

Note that $\chi_{C_k}\to\chi_C$ in $L^1$, $\haus^{n-1}(\de C_k)\to\haus^{n-1}(\de C)$ and $|\mu_{C_k}|\overset{*}{\rightharpoonup}|\mu_C|$ as $k\to\infty$. Indeed, since clearly $C_k\subset C$ and $h(C_k,C)\le h(E_k,E)$, we have
\begin{equation*}
\haus^n(C_k)\subset\haus^n(C)\subset\haus^n(C_k+\lambda_k\mathbb{B}^n)
\end{equation*}
and, by~\eqref{eq:perim_ineq},
\begin{equation*}
\haus^n(\de C_k)\subset\haus^n(\de C)\subset\haus^n(\de(C_k+\lambda_k\mathbb{B}^n)),
\end{equation*}
so that $\chi_{C_k}\to\chi_C$ in $L^1$ and $\haus^{n-1}(\de C_k)\to\haus^{n-1}(\de C)$ by the Steiner formulas for outer parallel bodies. Moreover, for each $0\le t\le h$, let 
\begin{equation*}
(C_k)_t=C_k\cap \set*{-t\nu+\nu^\perp}, \qquad C_t=C\cap \set*{-t\nu+\nu^\perp}.
\end{equation*}
Clearly $\chi_{(C_k)_t}\to\chi_{C_t}$ in $L^1(-t\nu+\nu^\perp)$ for each $0\le t\le h$ by Lemma~\ref{lemma:approx}. Thus, by the divergence theorem and Tonelli theorem, we have
\begin{equation*}
\int_{\R^n}\phi\ d\mu_{C_k}=\int_{C_k}\nabla\phi\ dx=\int_0^h\int_{(C_k)_t}\nabla\phi\ dx'dt\to\int_0^h\int_{C_t}\nabla\phi\ dx'dt=\int_C\nabla\phi\ dx=\int_{\R^n}\phi\ d\mu_C
\end{equation*}
for all $\phi\in\mathscr{C}^1_c(\R^n)$, and $|\mu_{C_k}|\overset{*}{\rightharpoonup}|\mu_C|$ follows as in Lemma~\ref{lemma:approx}. 
Thus, since $\Phi$ is continuous, by~\cite{maggi12}*{Theorem~20.6} we get $P_\Phi(C_k)\to P_\Phi(C)$.

In conclusion, since $g_\nu$ is continuous and $r_k\to r$, $\haus^{n-1}(E_k)\to\haus^{n-1}(E)$ as $k\to\infty$, passing to the limit in~\eqref{eq:lemma_cones-k} as $k\to\infty$, we find
\begin{equation*}
P_\Phi(C)\ge\haus^{n-1}(W_\nu)r^{n-2}g_\nu(h,r)+\Phi(-\nu)\,\haus^{n-1}(E),
\end{equation*}
which immediately gives~\eqref{eq:lemma_cones}. The proof of Lemma~\ref{lemma:cones} is thus complete. 
\end{proof}

We are now ready to prove our main result.

\begin{proof}[Proof of Theorem~\ref{th:main}]
Since $A$ and $B$ are closed sets and $A\subset B$, the distance $h(A,B)$ is given by
\begin{equation*}
h(A,B)=\max_{y\in B}\min_{x\in A}|x-y|.
\end{equation*}
Let $a\in A$ and $b\in B$ be such that $h(A,B)=|a-b|$. It turns out that $b\in B\setminus A$ and that $a$ is the orthogonal projection of $b$ onto the closed convex set $A$. By definition of the half-space $H$ in~\eqref{eq:main_defs_r_H_nu_H} and by minimality of the projection, the closed hyperplane
\begin{equation*}
\de H=\set{x\in\R^n : \scalar*{b-a,x-a}=0}
\end{equation*} 
is a supporting one for the convex set $A$ in the point $a$. 

Since $A\subset B\cap H$ and $B\cap H\subset B$, by the monotonicity formula~\eqref{eq:phi-perim_ineq} we have
\begin{equation*}
P_\Phi(A)\le P_\Phi(B\cap H)\le P_\Phi(B),
\end{equation*}
therefore
\begin{align}\label{eq:stima_proof}
\delta_\Phi(B,A)&=\delta_\Phi(B,B\cap H)+\delta_\Phi(B\cap H,A)\ge\delta_\Phi(B,B\cap H)=P_\Phi(B)-P_\Phi(B\cap H)\nonumber\\
&=\int_{\de B\cap H^c}\Phi(\nu_B)\ d\haus^{n-1}-\Phi(\nu_H)\,\haus^{n-1}(B\cap\de H).
\end{align}
where $H^c=\R^n\setminus H$ for brevity.

Let us now set $C=\mathcal{C}(b,B\cap\de H)$. Note that $C\subset B\cap\closure{H^c}$, thus by~\eqref{eq:phi-perim_ineq} we have
\begin{equation}\label{eq:cono1_dim}
P_\Phi(C)\le P_\Phi(B\cap\closure{H^c}).
\end{equation}
We set $C_{lat}=\mathcal{C}(b,\de B\cap H)$. Then $\de C=C_{lat}\cup (B\cap \de H)$ and
\begin{equation}\label{eq:cono2_dim}
P_\Phi(C)=\int_{C_{lat}}\Phi(\nu_C)\ d\haus^{n-1}+\Phi(-\nu_H)\,\haus^{n-1}(B\cap\de H).
\end{equation} 
Moreover
\begin{equation}\label{eq:cono3_dim}
P_\Phi(B\cap\closure{H^c})=\int_{\de B\cap H^c}\Phi(\nu_B)\ d\haus^{n-1}+\Phi(-\nu_H)\,\haus^{n-1}(B\cap\de H).
\end{equation}
Therefore, combining~\eqref{eq:cono1_dim}, \eqref{eq:cono2_dim} and~\eqref{eq:cono3_dim}, we find
\begin{equation}\label{eq:cono_fine}
\int_{C_{lat}}\Phi(\nu_C)\ d\haus^{n-1}\le \int_{\de B\cap H^c}\Phi(\nu_B)\ d\haus^{n-1}.
\end{equation}

Finally, inserting~\eqref{eq:cono_fine} in~\eqref{eq:stima_proof}, we get
\begin{equation}\label{eq:final_ineq}
\delta_\Phi(B,A)\ge\int_{C_{lat}}\Phi(\nu_C)\ d\haus^{n-1}-\Phi(\nu_H)\,\haus^{n-1}(B\cap\de H).
\end{equation}

Up to a translation, we can now assume that $a=0$ and apply Lemma~\ref{lemma:cones} to the cone~$C$. We thus have
\begin{equation}\label{eq:apply_lemma}
\int_{C_{lat}}\Phi(\nu_C)\ d\haus^{n-1}\ge\haus^{n-1}(W_{\nu_H})r^{n-2}g_{\nu_H}(h,r),
\end{equation}
where $h=\dist(b,B\cap\de H)=|b-a|=h(A,B)$ and $r>0$ is such that $r^{n-1}\haus^{n-1}(W_{\nu_H})=\haus^{n-1}(B\cap\de H)$. Inserting~\eqref{eq:apply_lemma} in~\eqref{eq:final_ineq}, we find
\begin{align*}
\delta_\Phi(B,A)&\ge\haus^{n-1}(W_{\nu_H})r^{n-2}g_{\nu_H}(h,r)-\Phi(\nu_H)r^{n-1}\haus^{n-1}(W_{\nu_H})\\
&=\haus^{n-1}(W_{\nu_H})r^{n-2}\big(g_{\nu_H}(h,r)-\Phi(\nu_H)r\big)
\end{align*}
and the proof of Theorem~\ref{th:main} is complete.
\end{proof}


\end{document}